\documentclass[11pt]{amsart}
\usepackage{amsmath}
\usepackage{amsthm}
\usepackage{amssymb}
\usepackage{amscd}
\usepackage{amsfonts}
\usepackage{amsbsy}
\usepackage{amsfonts,amssymb,amstext,amsthm}
\usepackage{epsfig}
\include{psfig}

\def\be{\begin{equation}}
\def\ee{\end{equation}}

\def\bq{\begin{eqnarray}}
\def\eq{\end{eqnarray}}

\def\beq{\begin{eqnarray*}}
\def\eeq{\end{eqnarray*}}

\def\a{{\alpha}}
\def\ka {\kappa }
\def\ph{{\varphi}}
\def\b{{\beta}}
\def\De{{\Delta}}

\def\d{\delta}
\def\p{\partial }
\def\g{{\gamma}}
\def\G{{\Gamma}}
\def\t{{\theta}}
\def\s{\sigma }
\def\S{\Sigma }
\def\e{{\varepsilon}}
\def\o{\overline}
\def\lc {limit cycle }
\def\lcs {limit cycles }
\def\nbd {neighborhood }
\def\nbds {neighborhoods }
\def\pmap{Poincar\' e map }

\def\tes {there exists }
\def\st{such that }

\def\gzt{Growth-and-Zeros Theorem }
\def\gi{Gronwall inequality}

\newtheorem{Thm}{Theorem}
\newtheorem{Def}[Thm]{Definition}
\newtheorem{Lem}[Thm]{Lemma}
\newtheorem{Prop}[Thm]{Proposition}
\newtheorem{Cor}[Thm]{Corollary}

\newcommand{\dd}{\mathbb{D}}

\newcommand{\R}{\mathbb{R}}
\newcommand{\C}{\mathbb{C}}
\newcommand{\sss}{\mathbb{S}}

\newcommand{\cc}{\mathbb{\cal C}}

\def\vf {vector field }
\def\vfs {vector fields }
\def\qvf {quadratic vector field }
\def\qvfs {quadratic vector fields }
\def\l {\lambda }
\def\L {\Lambda }
\def\rr{\mathbb{R}}
\def\cc{\mathbb{C}}
\def\zz{\mathbb{Z}}
\def\la{\lambda}

\begin{document}

\title[Hilbert's 16th problem for quadratic vector fields] {A restricted version of the Hilbert's 16th problem for quadratic vector fields}

\author[Yu Ilyashenko and Jaume Llibre]
{Yu Ilyashenko$^1$ and Jaume Llibre$^2$}

\address{$^1$ Cornell University, US; Moscow State and Independent
Universities, Steklov Math. Institute, Moscow.}

\email{yulij@math.cornell.edu}

\address{$^2$ Departament de Matem\`{a}tiques, Universitat Aut\`{o}noma de Barcelona, 08193 Bellaterra, Barcelona, Catalonia, Spain.}

\email{jllibre@mat.uab.cat}

\thanks{The first author was supported by part by the grants NSF
0700973, RFBR-CNRS 050102801, RFBR 07-01-00017-\`{a}. The second
author is partially supported by a MCYT/FEDER grant number
MTM2008--03437 and by a CIRIT grant number 2005SGR 00550.}

 \dedicatory{} \commby{}
\date{}
\dedicatory{}

\maketitle

\begin{abstract}
The restricted version of the Hilbert 16th problem for quadratic
vector fields requires an upper estimate of the number of limit
cycles through a vector parameter that characterizes the vector
fields considered and the limit cycles to be counted. In this paper
we give  an upper estimate of the number of limit cycles of
quadratic vector fields $``\sigma $--distant from centers and $\ka
$-distant from singular quadratic vector fields'' provided that the
limit cycles are $``\delta $--distant from singular points and
infinity''.
\end{abstract}

\section{Introduction and statement of the main result}\label{s1}

Hilbert 16th problem asks (see \cite{Hi}): {\it what may be said
about the number and location of limit cycles of a polynomial vector
field of degree $n$ in the real plane?} The main contributions in
this direction were the works of \'{E}calle \cite{Ec} and Ilyashenko
\cite{Il} who proved that any polynomial vector field has finitely
many limit cycles, and also the work of Llibre and Rodr\'{\i}guez
\cite{LR} who showed that any finite location of limit cycles is
realized by a polynomial vector field of a convenient degree. But
the complete answer to Hilbert 16th problem is unknown even for $n =
2.$ Even the existence of an uniform upper bound of the number of
limit cycles for quadratic vector fields (polynomial vector fields
of degree $2$) is not yet proved. Limit cycles of a quadratic vector
field may surround only one singular point, and it is of type focus
(for more details see \cite{Ye}). Moreover, quadratic vector fields
have at most two foci (see again \cite{Ye}). Limit cycles
surrounding the same singular point form a nest. Recently Zhang
Pingguang \cite{ZP1, ZP2, ZP3} proved that only one nest of
quadratic vector field may have more than one limit cycle.

The restricted version of the Hilbert 16th problem for quadratic
vector fields allows us to introduce a vector parameter that
characterizes the vector field and the limit cycles to be counted.
The upper bound for the number of limit cycles is expressed through
this parameter.

In this paper we give {\it an upper estimate of the number of limit
cycles of quadratic vector fields ``$\sigma $--distant from centers
and $\ka $--distant from singular quadratic vector fields'' provided
that the cycles are $``\delta $--distant from singular points and
infinity''.} The precise sense of assumptions in quotation marks
is explained below. The upper estimate mentioned above depends on
$\sigma , \ka $ and $\delta .$

\subsection{Normalized \qvfs }

We consider quadratic vector fields with a focus point $0$ and
estimate the number of limit cycles that surround this point. The
system has the form
\begin{equation}   \label{eqn:num1}
\dot z =\mu z + Az^2 + Bz\bar z + C\bar z^2,
\end{equation}
where $\mu ,A, B, C$ are complex coefficients. Rescaling: $z \mapsto
cz$ and $t \mapsto c't$, $c \in \mathbb C$, $c' \in \mathbb R$ brings
it to
$$
\dot z = c'(\mu z + Acz^2 + B \bar c z\bar z + C\frac {{\bar
c}^2}{c}{\bar z}^2).
$$
Hence, after an appropriate normalization, we can take in
\eqref{eqn:num1}: $\mu = \lambda_1 + i, \ \max (|A|, |C|) \le 1, |B|
\le 2$. Moreover, the normalized tuple $(A,B, C)$ has the form:
either $A =1$ and $|B| \le 2, \ |C| \le 1$, or $B = 2$ and $|A| \le
1, |C| \le 1$, or $C= 1$ and $|A| \le 1, |B| \le 2$. The reason for
distinguishing $B$ will be seen later.

To summarize, the normalized quadratic \vf has the form:
\begin{equation}   \label{eqn:n1}
\dot z = \mu z + z^2 + Bz\bar z + C{\bar z}^2, \ |B| \le 2, |C| \le 1,
\end{equation}
or
\begin{equation}   \label{eqn:n2}
\dot z = \mu z + Az^2 + 2z\bar z + C{\bar z}^2, \ |A| \le 1, |C| \le 1,
\end{equation}
or
\begin{equation}   \label{eqn:n3}
\dot z = \mu z + Az^2 + Bz\bar z + {\bar z}^2, \ |A| \le 1, |B| \le 2,
\end{equation}
with $\mu = \l_1 + i$.

Moreover, we consider that $\l_1 \ge 0$. If not, we reverse the time
and make a symmetry $z \mapsto \bar z.$

The tuple of parameters $(\l_1, A, B, C)$ with $(A,B, C)$ normalized
as above is denoted by $\l $ and the corresponding vector field (and
equation) is denoted by $v_\l$. The space of all these $\l $'s is
denoted by $\L $. It is homeomorphic to the glued union of three
copies of $\R^+ \times \dd^2\times \dd^2$, where $\R^+= [0,\infty)$
and $\dd^2= \{z\in \C : |z|\le 1\}$; the gluing maps identify the
boundaries of the cells $\R^+ \times \dd^2\times \dd^2$; we will not need
these maps below.

\subsection{Center conditions}

Center conditions for quadratic \vfs are well known; see the works
of Dulac \cite{Du}, Kapteyn \cite{K1, K2}, Bautin \cite{Ba}. In the
complex form they were found by Zoladek \cite{Zo}, see the next
theorem. We will use the latter form of the center conditions.

\begin{Thm}\label{thm:t1}  A quadratic \vf \eqref{eqn:num1} has a center at zero
if and only if the following holds:
\begin{equation}  \label{eqn:cen}
\begin{array}{l}
g_1 := \l_1 = 0,\\
g_2 := {\rm Im}(AB) = 0,\\
g_3 := {\rm Im}[(2A + \bar B)(A -2\bar B)\bar BC] = 0,\\
g_4 := {\rm Im}[(2A + \bar B)(|B|^2 - |C|^2){\bar B}^2 C] = 0.
\end{array}
\end{equation}
\end{Thm}

\begin{Def}\label{def:cent}
A normalized \qvf is called $\sigma$--distant from centers provided that
\begin{equation}\label{eqn:distcen}
\sum_{j=1}^4 |g_j(\l )| \ge \s .
\end{equation}
The set of all $\l \in \L$ for which $v_\l$ is $\sigma$--distant from
centers is denoted by $\L(\s).$
\end{Def}

\subsection{$\delta $--tame limit cycles}

Now for any $\d > 0$ we define $\d $--tame limit cycles of a
normalized quadratic vector field. Note that the normalization of a
quadratic vector field provides a scale on the phase plane. Thus the
following definition makes sense. For any $\d \in (0,1)$ and any $\l
\in \L $ denote by $B(\l,\d )$ the   disc $|z| \le {\d }^{-1}$ in
$\C$ minus all the open $\d $-\nbds of the singular points of
$v_\l$, both real and complex, except for the point $0$.

\begin{Def}  A \lc of a normalized \vf is called \emph{$\d $--tame} provided
that it belongs to $B(\l ,\d )$.
\end{Def}

\subsection{Singular \qvfs }

A \qvf with a focus at the zero and a line of singular points not passing through zero is called
{\it singular}. Such a \qvf can be written as
$$
\dot z =\mu zl(z),
$$
where $l(z)$ is a real polynomial of degree $1$ of the form $l(z) = az + \bar a\bar z + 1$.
After normalization, this equation has the form
\begin{equation}  \label{eqn:s1}
\dot z = \mu z + z^2 +\frac {\mu }{\bar \mu }z\bar z := v_s(z),
\end{equation}
where $\mu= \la_1 + i$. The $s$ of $v_s$ is for a singular
quadratic vector field. Any normalized \qvf close to a singular one
has the form
\begin{equation}  \label{eqn:deco}
v = v_s + u, \qquad u =bz\bar z + c{\bar z}^2;
\end{equation}
here $v_s$ is singular, $b$ and $c$ are small. In this expression
for $v$, its coefficient $B$ may be greater in modulus than $1$ but
smaller than $2$ because $b$ is small. Still the normal form of
$v$ is \eqref{eqn:n1}. That is why $B$ is distinguished in the
definition of the normal form. To find decomposition
\eqref{eqn:deco} for a \vf $v$ in the form \eqref{eqn:n1}, take
$v_s$ as in \eqref{eqn:s1} and $u$ as in \eqref{eqn:deco} with
coefficients:
$$
b = B - \frac {\mu }{\bar \mu }, \ c = C.
$$
Let $\| \cdot \|_2$ denote the $L_2$ norm on a circle.   Then
$$
||r^{-2}u||_2^2 = {|b|}^2 + {|c|}^2.
$$

\begin{Def}
A \qvf is \emph{ $\ka $--distant} from the set of singular \qvfs if
$||r^{-2}u||_2 > \ka $ in \eqref{eqn:deco}.
\end{Def}

\subsection{Main result}

\begin{Thm}[Main Theorem]\label{thm:mt2}
For any $\{ \delta , \sigma , \kappa \} \subset (0,0.1)$, the number of $\d $--tame
\lcs of a normalized \qvf which is $\s $--distant from centers and $\ka $--distant
from singular quadratic vector fields is no greater than
$$
H(2, \d , \s , \ka ) = |\log \sigma |\exp (\exp ({10}^{25}
\delta^{-31}\kappa^{-2})).
$$
\end{Thm}
This estimate is irrealistic but this is the only known estimate of
this kind.

This paper is the first in a series of papers aimed to estimate the
number of $\d $--tame \lcs of quadratic vector fields. In a
subsequent paper we prove that for $\ka $ sufficiently small: $\ka
\le \ka_0 (\delta, \s)$, the \vf \eqref{eqn:deco} has only one $\d
$--tame limit cycle.   A similar result, without a quantitative
estimate on the value of $\kappa_0(\delta , \sigma )$, is obtained
in the preprint \cite{DR}.

\subsection{Growth--and--Zeros Theorem}

Limit cycles correspond to the fixed points of the \pmap . For
normalized \qvf $v_\l $ consider the \pmap $P_\l $ of a segment of a
positive semiaxes $\mathbb R^+$ with the left endpoint $0$ into
$\mathbb R^+$; the right endpoint will be specified later.

The number of the fixed points of this map will be estimated with
the use of the theorem named in the title of the subsection; for its
proof see \cite{Il}, \cite{IP}.

\begin{Thm} Let $U\subset \C$ be a connected and simply connected domain
and $K\subset U$ be a path connected compact set. Let $D$ be the
internal diameter of $K$, and
$$
\text {gap }(K,U):= \rho (K, \partial U) \ge \e ,
$$
where $\rho (K, \partial U) = \displaystyle \min_{a\in K, b\in
\partial U}|a - b|.$ Let $f: {\o U} \to \C$ be a holomorphic
function. Then
\begin{equation}\label{eqn:gaz}
\# \{ z \in K\mid f(z) = 0\} \le B_{K,U}(f)\exp\left(\frac {2D}{\e}\right),
\end{equation}
where
\begin{equation}   \label{eqn:bern}
B_{K,U}(f) = \log \frac {\max_{\o U}|f|}{\max_K |f|}.
\end{equation}

\end{Thm}

As usual $\o U$ denotes the closure of $U$. The expression
$B_{K,U}(f)$ is called {\it the Bernstein index of $f$ for $K,U$}.
The exponential in \eqref{eqn:gaz} is called \emph{the geometric
factor.}  We will often write:
$$
M = \max_{\bar U} |f|, \qquad m = \max_K |f|.
$$
This theorem will be applied to bound the number of zeros of the
displacement function
$$
f_\la = P_\la- id
$$
of the Poincar\'{e} map  $P_\la$ of $v_\la $;    these zeros
correspond to limit cycles of $v_\la .$

There are the following steps in the application of this theorem:

- choice of $K$ and finding the lower estimate for $m = \displaystyle\max_K |f_\la|$;

- choice of $U$ and finding the upper estimate for $M = \displaystyle\max_{\bar U}|f_\la |$.

\section{The lower estimate of the maximum of the displacement}

\subsection{Normalized \qvfs in polar coordinates}

To write equation \eqref{eqn:num1} in polar coordinates $(r, \theta
)$ note that
$$
{(\log z)}^\cdot = \frac {\dot r}{r} + i\dot \theta  = \frac
{v(z)}{z}.
$$

Hence,
\begin{equation}\label{eqn:rad}
\begin{array}{l}
\dot r = r\mbox{Re }\dfrac {v(z)}{z} = r(\l_1 + rf_\l (\theta )),\\
\dot \theta = \mbox{Im }\dfrac {v(z)}{z} = 1 + rg_\l (\theta ),
\end{array}
\end{equation}
where $f_\l $ and $g_\l $ are trigonometric polynomials of degree $3$:
$$
f_\l (\theta ) = \mbox{Re } h_\l (\theta ),\qquad  g_\l (\theta ) =
\mbox{Im }h_\l (\theta ),
$$
\begin{equation}\label{eqn:hl}
h_\l (\theta ) = Ae^{i\theta } + Be^{-i\theta } + Ce^{-3i\theta }.
\end{equation}
For the normalized equations, $|h_\l (\theta )| \le 4$. Hence,
\begin{equation}     \label{eqn:mod}
|f_\l (\theta )| \le 4, \ |g_\l (\theta )| \le 4.
\end{equation}

\subsection{Compactification } \label{sub:com}

\begin{Lem} \label{lem:delta} If a system $v_\l$ has at least one $\d $--tame limit
cycle, then $|\l_1| \le 4/\d $.
\end{Lem}
\begin{proof}
Let $\l_1 > \frac 4 \d $. Recall that $r \le \d^{-1}$ in $B(\l , \d
)$. Then in $B(\l , \d )$, $\dot r \ge 0$    by \eqref{eqn:rad}
and \eqref{eqn:mod}. Hence, the \vf $v_\l $ has no limit cycles in
$B(\l , \d )$.
\end{proof}

\subsection{Complex extension of the Poincar\' e map near zero}

We will complexify nonautonomous equation corresponding to the
system \eqref{eqn:rad} making $r$ complex and denoting it by $w$ and
keeping $\theta $ real. We get:
\begin{equation}\label{eqn:radc}
\frac {dw}{d\theta } = w\frac {\l_1+wf_\l (\theta )}{1+wg_\l (\theta
)}:= F_\l (w,\theta ), \quad w \in \mathbb C, \quad \theta \in \sss^1.
\end{equation}
Recall that $||f_\l || \le 4$ and $||g_\l || \le 4$. When the norm
is not specified, it is the $C$--norm of a function on the circle.

For any value of $\l_1$, we will find $R$ and $\e $ in such a way
that the orbit that starts in   a cross--section $D_\e := \{ |w|
\le \e \} \times \{ 0\} $ keeps inside $W:= \{|w| \le R \} \times
\sss^1$ when $\theta $ ranges over $[0,2\pi ]$. We call this \emph
{property (*) of \eqref{eqn:radc}.}

\begin{Lem}\label{lem:radc}
Equation \eqref{eqn:radc} satisfies property (*) for $R = 0.01$ and
\begin{equation}     \label{eqn:epsnew}
\e = 2 \e(\l) = \begin{cases} 0.001 \mbox{ for }\l_1 \in [0,0.1]\\
R e^{-4\l_1\pi }\mbox{ for } \l_1 >0.1.
\end{cases}
\end{equation}
\end{Lem}

\begin{proof} The proof is based on the Gronwall inequality that
measures the divergence of two solutions of a differential equation.
To apply the classical Gronwall inequality to a differential
equation with the complex phase space, we simply take the
realification of this space. In case when one of the solutions is
identically zero, the Gronwall inequality measures the norm of the
other solution. For equation \eqref{eqn:radc} this inequality has
the folowing form. Let
$$
L = \max_W\left| \frac {\p F_\l }{\p w}\right| ,
$$
and $|w(0)| \le \e $. Then the Gronwall inequality claims that
\begin{equation}\label{eqn:ron}
|w(\theta )| \le \e e^{L\theta } \mbox{ for } \theta \in [0,2\pi ],
\end{equation}
provided that
\begin{equation}\label{eqn:eps}
\e e^{2\pi L} \le R.
\end{equation}

To get an upper bound for $L$, note that
\begin{equation}    \label{eqn:deriv}
\frac {\p F_\l }{\p w} = \frac {\l_1+2wf_\l }{1+wg_\l } - \frac
{w(\l_1+wf_\l )}{{(1+wg_\l) }^2}g_\l .
\end{equation}

Note that $||f_\l || \le 4, \ ||g_\l || \le 4$. Hence,
\begin{equation}   \label{eqn:lip}
L \le \begin{cases} 0.2 \mbox{ for } \l_1 \le 0.1,\\
2\l_1\mbox{ for } \l_1 > 0.1.
\end{cases}
\end{equation}
Now, \eqref{eqn:eps} yields Lemma \ref{lem:radc}.
\end{proof}

\subsection{Description of $K_\l $} Let $\Gamma $ be the positive $x$
semiaxis. Assume that system $v_\lambda$ has no $\d$--tame limit
cycles around the origin. Then Theorem \ref{thm:t1} holds for this
system. In what follows, we consider the opposite case. Let
$a(\la )$ be the intersection point of the outmost tame limit cycle
surrounding the origin with $\Gamma$. Let $s_\la $ be the segment
$[0,a(\la )],$ and $\e (\lambda )$ be the same as in
\eqref{eqn:epsnew}.

\begin{Lem}[First Main Lemma] \label{lem:main1}
For the set
\begin{equation}\label{eqn:kl}
K_\l = s_\l \cup D_{\e(\l)}
\end{equation}
the following lower estimates hold:
\begin{equation}\label{eqn:ml}
m(\l ) := \max_{w\in K_\l }|P_\l (w) - w| \ge  10^{-26}\sigma \mbox{
for } \l_1 \le 0.1 \mbox{ and }
\end{equation}
\begin{equation}     \label{eqn:22}
m(\l ) \ge {10}^{-26/\d } \mbox { for } \l_1 > 0.1.
\end{equation}
\end{Lem}

Note that these estimates do not depend on $\ka $. The lemma is
proved in the next five subsections.

\subsection{Proof of Lemma~\ref{lem:main1} for the case of a strong focus}

In this subsection when we say that the normalized \qvf has a strong focus we mean that $\l_1 > 0.1$.

To prove Lemma \ref{lem:main1} in this case, we use the reversed Cauchy
inequality for the first derivative: if $f$ is holomorphic in a disc
$D_\e = \{|w| <\e\}\times \{0\}$ and continuous on the boundary of this disc, then
\begin{equation}\label{eqn:cauc}
\max_{D_\e }|f| \ge \e |f'(0)|.
\end{equation}

For $f = P_\l (w) - w$, and in the case $\l_1 >0.1$, we have:
$$
f'(0) = e^{2\pi \l_1} - 1 > 0.3e^{2\pi \l_1}
$$

By Lemma \ref{lem:radc}, $f$ is well defined in $D_\e $ for $\e =
0.005e^{-4\l_1\pi }$. Hence
$$
m \ge \max_{D_\e }|f| \ge 0.0015e^{-2\l_1\pi } \ge e^{-26/\d },
$$
where the last inequality follows from $\l_1 \le \frac 4
{\delta}$ and $\delta < 0.1$. This yields \eqref{eqn:22} and
proves Lemma~\ref{lem:main1} for $\l_1
> 0.1$. To prove this lemma for $\l_1 \le 0.1$, we need first to
study the case $\l_1 = 0$ and then to perturb it.

\subsection{Seven--jet of the Poincar\' e map for linear part a center}

The \pmap for the point zero of the normalized   \qvf $v_\l $ may
be decomposed in a Taylor series
\begin{equation}\label{eqn:poin}
P_\la(w) = \sum_{j\ge 1} a_j(\la )w^j.
\end{equation}
This series converges at least in a \nbd of the form $D^0= \{ |w|
\le r_0\}$ for a convenient $r_0>0$. Consider the case $\l_1 = 0$.
For such $\l $, the coefficients $a_j(\l)$ become functions only of
$(A,B,C)$ not necessarily   normalized.

\begin{Lem}\label{lem:poin}
Let $\lambda_1 = 0$. Then for the decomposition \eqref{eqn:poin},
\[
\begin{array}{l}
a_1 \equiv 1, \ a_2 \equiv 0, \ a_3 = \a_0g_2, \ a_4 =\a_1g_2,\\
a_5 = \b_0g_3 + \b_1g_2, \ a_6 =\b_2g_3 +\b_3g_2, \ a_7 =\g_0g_4
+\g_1g_3 +\g_2g_2,
\end{array}
\]
where $g_2, g_3, g_4$ are the polynomials from the center conditions \eqref{eqn:cen}, $\alpha_j, \beta_j, \gamma_j$ are polynomials in the variables $A, B, C$, and $\alpha_0, \beta_0, \gamma_0 $ are constant.
Moreover, on the set of $\l =(0, A, B, C)$ with the tuples $A, B, C$ normalized we have:
\[
\begin{array}{l}
|g_2| \le 2, \quad |g_3| \le 30, \quad |g_4| \le 36;\\
\\
|\a_0| = 2\pi, \\
\\
|\b_0|= \dfrac{2\pi}{3}, \quad |\b_1|\le \dfrac{2\pi}{9}(284+108 \pi
):= B_1 < 500, \\
\\
|\g_0|= \dfrac{5\pi}{4}, \quad  |\g_1|\le \dfrac{\pi}{72} (5816+1536 \pi ):= C_1 < 500, \\
\\
|\g_2|\le \dfrac{\pi  \left(5019144+2565120 \pi +345600 \pi
^2\right)}{1080}:= C_2 \in [4\cdot {10}^4, {10}^5].
\end{array}
\]
Expressions for the $\a$'s, $\b$'s and $\g$'s are given in the
appendix.
\end{Lem}

Lemma \ref{lem:poin} has been proved using the algebraic manipulator
mathematica and the three normal forms for the quadratic vector
fields. The algorithm is sketched in the appendix.

\subsection{Lower estimate: case of a linear part a center}

Denote the normalized tuple $\l$ with $\l_1 = 0$ by $\l':=
(0,A,B,C)$.   Recall that in \eqref{eqn:kl}, $\e (\l') = 0.0005.$
Let $m(\l)$ be the same as in \eqref{eqn:ml}.   Recall that
$\Lambda (\sigma )$ appears in Definition~\ref{def:cent}. The next
lemma is one of the main steps in the proof of Theorem~\ref{thm:t1}.

\begin{Lem}\label{lem:mcen}
For the normalized   $\l \in \Lambda (\sigma )$ with $\l_1 =0$, we
have:
$$
m(\l ) \ge 2\cdot {10}^{-23}\sigma := m_0.
$$
\end{Lem}

\begin{proof} Let $f_\l = P_\l - id$.  By Lemma~\ref{lem:poin}, for $\l_1 =0,
 \l = \l'$, we have:
$$
f_\l (0) = f'_\l (0) = f''_\l (0) = 0.
$$
For vector fields $\s $--distant from centers, we will prove a lower
estimate:
$$
|a_j(\l )| \ge m_j(\s ),
$$
with $m_j$ explicitly written for at least one $j \in \{ 3; 5; 7\}$.
By Lemma~\ref{lem:radc} the function $f_\l $ is holomorphic in the
disc $|w| \le 0.001:= 2 \e (\l')$. Hence, \tes  $j \in \{3, 5, 7\} $
\st
$$
m(\l ) \ge \max_{D_{\e (\l')}}|f_\l |\ge m_j(\s ) \cdot  \e
{(\l')}^j.
$$

The lower bounds for $a_j$ are found in the following way. For $\a ,
\b \in (0,1)$ chosen later, the compact set $\L_0(\sigma ) = \L
(\sigma ) \cap \{\l_1 = 0\}$ is split into three parts $\Sigma_2,
\Sigma_3, \Sigma_4$   where
$$
\Sigma_2 =(|g_2| \ge \a \s ), \,\, \Sigma_3 = (|g_2| + |g_3| \ge \b
\s )\setminus \Sigma_2, \,\, \Sigma_4 = \Lambda_0(\s ) \setminus
(\Sigma_2 \cup \Sigma_3).
$$
On $\Sigma_j, \ |a_{2j-1}(\l )|$ is estimated from below. By
Lemma~\ref{lem:poin}, on $\S_2$, $|a_3| \ge \a_0\a \s $.  Let $B_1,
C_1, C_2$ be the same as in Lemma~\ref{lem:poin}.   On $\S_3$ we
have:
$$
a_5 =\b_0g_3 +\b_1g_2.
$$
Hence,
$$
\left| a_5\left|_{\S_3} \right. \right| \ge \b_0(\b - \a )\s - B_1\a \s =
\b_0\left(\b - \a \left(1 + \frac{B_1}{\b_0}\right)\right)\s .
$$

If we choose $\a $ so small in comparison with $\b $ that
\begin{equation}  \label{eqn:alpha}
\a \left(1 + \frac {B_1}{\beta_0}\right) \le \frac \b 2,
\end{equation}
then
$$
\left| a_5\left|_{\S_3} \right. \right| \ge \frac {\b_0\b \s }{2}.
$$
On $\S_4$ we have:
$$
a_7 = \g_0g_4 +\g_1g_3 +\g_2g_2.
$$

As $C_2 > C_1$, we have:
$$
\left| a_7\left|_{\S_4} \right. \right| \ge \g_0(1 - \b )\s - C_2\b \s =
\g_0\left(1 - \b \left(1 +\frac{C_2}{\g_0}\right)\right)\s .
$$
If $\b $ is so small that
\begin{equation}    \label{eqn:beta}
\b \left(1 + \frac {C_2}{\g_0}\right) \le \frac 1 2,
\end{equation}
then
$$
\left| a_7\left|_{\S_4} \right. \right| \ge \frac {\g_0\s }{2}.
$$
Now,
$$
\begin{array}{l}
\left| m\left|_{\S_4} \right. \right| \ge \displaystyle\min_{\S_4}|a_7|r^7_0 \ge \dfrac {\g_0\s }{2} \e (\l')^7:= m_4\s \\
\left| m\left|_{\S_3} \right. \right|  \ge \displaystyle\min_{\S_3}|a_5|r^5_0 \ge \dfrac {\b_0\b \s }{2} \e (\l')^5:= m_3\s \\
\left| m\left|_{\S_2} \right. \right|  \ge
\displaystyle\min_{\S_2}|a_3|r^3_0 \ge \a_0\a \s  \e (\l')^3:= m_2\s
.
\end{array}
$$
Due to Lemma~\ref{lem:poin}, inequalities \eqref{eqn:alpha},
\eqref{eqn:beta} hold for   $\beta = {10}^{-5}, \alpha = 2\cdot
{10}^{-8}$. Again by Lemma~\ref{lem:poin}, $m_2 > m_3 > m_4 > 2
\cdot {10}^{-23}$. This proves Lemma~\ref{lem:mcen}.



\end{proof}

\subsection{Proof of the First Main Lemma in case of the moderate focus}

  Recall that $m_0$ is the lower estimate of $\max_{D_\e }|P -
id|$ mentioned in Lemma~\ref{lem:mcen}. Here we consider the case
$\l_1 \in [m_0, 0.1]$. In this case, by Lemma~\ref{lem:radc}, the
displacement $f_\l $ of the \pmap is holomorphic in a disc $|w| \le
\e = 0.0005$. We have:
$$
|f'_\l (0)| \ge e^{2\pi m_0} - 1 \ge 2\pi m_0.
$$
Hence,
$$
\max_{|w| \le \e }|f_\l | \ge 0.003m_0.
$$

This proves the First Main Lemma in the case considered.

\subsection{Proof of the First Main Lemma in case of the slow focus}

We consider here the last remaining case $\l_1 \in (0,m_0]$, where
$m_0$ is the same as in Lemma~\ref{lem:mcen}, i.e. $m_0 = 2
\cdot 10^{-23}\sigma $. This case is treated as a small perturbation
of the case $\l_1 = 0$. Consider two systems \eqref{eqn:radc}
corresponding to $\l_1 = 0$ and $\l_1 \in (0,m_0]$ fixed. Let their
right hand sides be $F$ and $G$. We assume that $G$ corresponds
to a normalized \qvf which is $\sigma $-distant from centers. This
implies that $F$ corresponds to a similar field which is at least
$0.9 \sigma $-distant from centers. Let
$$
\max_W|F - G| < \Delta ,
$$
$$
\max_W\left| \frac {\p F}{\p w}\right| < L,
$$
where as before $W = \{|w| \le R\} \times \sss^1, \ R = 0.01$. Let
$\e = e^{-2\pi L}R$;   clearly, $m_0 < \frac R 2$. Then the
solutions $w_F$ and $w_G$ of the equations $\dfrac {dw}{dz} = F$ and
$\dfrac {dw}{dz} = G$ with the same initial condition $w(0): |w(0)|
< \frac \e 2$ diverge on the segment $0 \le \t \le 2\pi $  no more
than
\begin{equation}    \label{eqn:gronu}
|w_F(\t ) - w_G(\t )| \le 2\pi \De e^{2\pi L}.
\end{equation}
We apply \eqref{eqn:gronu} to our $F$ and $G$. We have:
$$
\De = \max_W \left| \frac {w\l_1}{1-wg_\l }\right| \le  \frac
{m_0}{96}
$$
in $W$. On the other hand, $L \le 0.2$ by \eqref{eqn:lip}. Hence,
for any two solutions $w_F$ and $w_G$ with the initial condition
$w(0)$ and $ |w(0)| \le \e (0) = 0.0005$, we have
$$
|w_F(2\pi ) - w_G(2\pi )| \le \frac {2\pi e^{0.4\pi }}{96}m_0 < 0.4
\ m_0.
$$

Suppose now that $w(0) = w_F(0)$ corresponds to the solution $w_F$
for which $|w_F(2\pi ) - w_F(0)| \ge 0.9 \ m_0$, and $w_G(0) =
w_F(0)$. Then $|w_G(2\pi) - w_G(0)| \ge \dfrac {m_0}{2}$, and
Lemma~\ref{lem:main1} is proved.

\section{Upper estimate of the displacement of the \pmap}
\label{sec:s3}

In this section we construct a \nbd $U_\l$ of the set $K_\l$ where
the \pmap $P_\l$ of equation $v_\l$ is well defined. We give a lower
estimate of the gap $\e$ between $K_\l$ and $\partial U_\l$, and
find an upper estimate for $f_\l = P_\l - \mbox{id}.$ To this end, we
find a universal gap between $\d$--tame \lcs of \qvfs that are
$\kappa$--distant from singular quadratic vector fields, and the curve $\dot \theta = 0.$

\subsection{The universal gap}\label{sub:univ}

A well known elementary property of \qvfs \eqref{eqn:num1} claims
that any closed orbit of these fields that surrounds the singular
point zero belongs to the domain $\dot \theta > 0$. It is a simple
consequence of the fact that any line has at most two contact points
with a quadratic vector field. The boundary of this domain is given
by the equation $r = -1/g_\l (\theta )$.

\begin{Lem}[Second Main Lemma]\label{lem:main2}
No $\d $--tame limit cycle of a normalized \vf $\ka $--distant from singular
quadratic vector fields intersects the curvilinear strip
$$
\Pi_\b = \left\{ (\t ,r) \in B_\l | \ r \in \left[ - \frac {1}{g_\l
(\t )} - \b , \ - \frac {1}{g_\l (\t )}\right] \right\} \,\, {\rm
for}\,\, \b = \frac {\d^{14}\ka }{{10}^{10}}.
$$
\end{Lem}
The proof of this lemma is technical. In the rest of this subsection
we make the first step of the proof that makes the existence of the
gap obvious. The estimates of the size of the gap are presented in
Section \ref{sec:bounds}.

Consider a zero isocline $\G $:
$$
\dot \theta = 0, \ r = -\frac {1}{g_\l (\t )}.
$$
The restriction of $\dot r$ to this isocline equals
$$
\dot r |_\G = \frac {H(v_\l )}{g^2_\l }, \qquad H(v_\l ) = \l_1g_\l
- f_\l .
$$

For the proof of Lemma~\ref{lem:main2}, we need a lower estimate of
$\left|H(v_\l )|_{\G \cap B_{\l ,\d }}\right|$. First of all, we
estimate from below the $L_2$--norm ${||H(v_\l )||}_2$ of $H(v_\l )$
on $\sss^1 = \rr / 2\pi \zz$. By \eqref{eqn:hl},
$$
H(v_\l )= \mbox{Im }\bar \mu h_\l .
$$

Note that $H(v_\l )$ is linear with respect to $v_\l $. Let $v_\l =
v_s + u_\l$ be the decomposition \eqref{eqn:deco} for $v_\l$. For
the singular \vf $v_s$  we have: $H(v_s) \equiv 0$. Hence,
$$
H(v_\l ) = H(u_\l ) = \mbox{Im }\bar \mu \tilde h_\l ,
$$
where $\tilde h_\l = be^{-i\t } + ce^{-3i\t }$.

Consider an arbitrary trigonometric polynomial $H$ on $\rr /2\pi \zz
$. If $H$ contains no complex conjugate monomials, that is, for any
entry $ae^{in\t } + be^{-in\t }$ at least one coefficient is $0$
(i.e. $ab = 0$), then
$$
{||\mbox{Im }H||}^2_2 = {||\mbox{Re }H||}^2_2 = \frac 1
2{||H||}^2_2.
$$
Indeed $H = \sum a_ne^{in\t }$ implies that $\mbox{Re}H = \frac 1
2(\sum (a_ne^{in\t } + \bar a_n e^{-n\t }))$, and consequently
${||\mbox{Re}H||}^2_2 = \dfrac 1 4 \sum (|a_n|^2 + |\bar a_n|^2) =
\frac 1 2{||H||}^2_2.$ The last conclusion holds because there are
no cancelations in the sum for $\mbox{Re }H$, by assumption.   The
same argument proves the statement for $\mbox{Im }H$.

\begin{Cor} \label{cor:1} For $v_\l $ which is $\ka $--distant from singular vector
fields ${||H (v_\l )||}_2 \ge \dfrac {|\mu |}{\sqrt 2}\ka $.
\end{Cor}

Indeed, for equations, $\kappa $-distant from singular ones, we have
 ${||H(v_\l )||}_2  = \frac {1}{\sqrt 2}|\mu |\sqrt
{b^2 + c^2} \ge \frac {|\mu |\kappa }{\sqrt 2}$.

We got therefore a uniform lower bound for the $L_2$--norm of the
restriction $\dot r|_\G$. It is now clear that a similar bound would
exist for $\min \dot r |_{\G \cap B(\d, \l)}$. Indeed, zeros of
$\dot r|_\G$ are located at the singular points of $v_\l$, and all
the points of $B(\d, \l)$ are at least $\d$--distant from these
points. After $\min \dot r|_{\Gamma \cap B(\lambda , \delta )}$ is
estimated, it is  easy to prove that the lower boundary of the
curvilinear strip $\Pi_\b $ has no contacts with the field $v_\l $.
>From this it follows that the $\d $-tame limit cycles can not
intersect $\pi_\b $. The detailed proof of Lemma \ref{lem:main2} is
completed in Section \ref{sec:bounds}.

\subsection  {Construction of the larger domain $U$ in the
\gzt}\label{sub:gap}

Let
$$
S_\l = s_\l \setminus D_{\e(\l)}
$$
and
$$
\bold D = B(\d, \l) \cap \left\{ r \le \frac {-1}{g_\l(\theta)}
- \beta \right\}
$$
For any $\l \in \Lambda $, consider a $(\b \d)/32$-\nbd $D'$
of the domain $\bold D \subset \mathbb R^+ \times \mathbb S^1$
in $\mathbb C \times S^1$. We will choose $\e $ in such a way that
any orbit of $v_\l$ that starts in $U_\e \times \{ 0\} $, where $U =
U_\e $ is the $\e$--\nbd of $S_\l$, stays in $D'$ while $\theta$
ranges in $[0, 2 \pi]$. Let
$$
L = \max_{D'} \left|  \frac {\partial F_\l}{\partial w}\right|.
$$
Then, by the \gi,
\begin{equation} \label{eqn:eps1}
\e = \frac {\b \d }{32}e^{-2 \pi L}
\end{equation}
should be the desired one. Indeed, the largest $\d $-tame \lc keeps
in $\bold D$ by Lemma~\ref{lem:main2}. Hence, all the orbits that
start on $S_\l \times \{ 0\} $, keep in $D$ by definition of $S_\l
$. Then, for $\e $ from \eqref{eqn:eps1}, the orbits that start in
$U_\e \times \{ 0\} $ would not quit $D'$ for $ \theta \in [0, 2\pi
]$. Moreover, they will be $\dfrac{\b \d }{32}$--close to the
real orbits starting at $S_\l$. Hence, the \pmap for $v_\l$ is well
defined in $U_\e$, and
$$
\max_{U_\e} |f_\l| = \max_{U_\e} |P_\l - \mbox{id}| \le \d^{-1} +
\frac {\b \d }{32}.
$$

By Lemma \ref{lem:radc}, the orbits that start in $D_{2 \e(\l)}$
stay in $D_R \times \sss^1$ as $\theta$ ranges over $[0, 2\pi].$ So,
the set $U_\l = U_\e \cup D_{2\e(\l)}$ is a \nbd of $K_\l$ in which
the \pmap of $v_\l$ is holomorphic, and
\begin{equation}     \label{eqn:max}
\max_{U_\l}|f_\l | = M \le \d^{-1} + 1.
\end{equation}

\subsection{The final estimate}\label{sub:fin}

We can now estimate the geometric factor in the \gzt.   For this
we need to get an upper bound for $L$, then a lower bound for $\e $.

To estimate $L$, we first get a lower estimate for the denominator
in the relation \eqref{eqn:deriv} for $\dfrac {\partial F}{\partial
w}$. We have:
$$
|w|_{D'} \le \d^{-1} + \frac {\beta \d }{32} << 2\d^{-1}.
$$
Now, estimate $\min_{D'}|l + wg_\l |$. If $(w, \theta ) \in D'$ is
such that $|g_\l (\theta )| \le \frac \d 4 $, then $|l + wg_\l | \ge
1 - \frac 2 \d \cdot \frac \d 4 \ge \frac 1 2$. Suppose that $|g_\l
(\theta )|$ is now greater than $\frac \d 4$. Find a point $(w',
\theta ) \in \bold D$ with $|w' - w| < \frac {\b \d }{32}$. Then
$$
|1 + wg_\l | \ge |\frac {1}{g_\l } + w'||g_\l | - |g_\l ||w' - w|
\ge \b \frac \d 4 - 4\frac {\b \d }{32} \ge \frac {\b \d }{8}.
$$
Hence,$\min_{D'}|1 + wg_\l | \ge \frac {\b \d }{8}$.

Moreover, by Lemma~\ref{lem:delta}, $\l_1 \le 4\d^{-1}$. Hence, by
\eqref{eqn:deriv},
$$
L \le 6145\d^{-3}\beta^{-2}.
$$
We substitute this $L$ into \eqref{eqn:eps1} and get the expression
for $\e $ through $\delta $ and $\beta $. Note that the expression
of $\beta $ through $\delta , \sigma , \kappa $ is given in
Lemma~\ref{lem:main2}.

The intrinsic diameter $D \le 2\d^{-1}$. Hence,
$$
\frac {2D}{\e } \le 128\d^{-2}\b^{-1}e^{({10}^5-2)\d^{-3}\b^{-2}}.
$$
This provides a double exponential estimate for the geometric factor
$\exp \dfrac {2D}{\e }$.

Note that for $\d < 0.1$ and $\b < 0.1$, increasing the factor
${10}^5 - 2$ in the exponential by one will compensate well the
division by the first factor. Finally,
\begin{equation}   \label{eqn:geom1}
\frac{2D}{\e } \le e^{({10}^5-1)\d^{-3}\b^{-2}}.
\end{equation}
We can now estimate the Bernstein index of $f_\l$. The numerator in
\eqref{eqn:bern} is estimated in \eqref{eqn:max}. The denominator is
estimated in the First Main Lemma (Lemma \ref{lem:main1}). We
replace it by even smaller value:
$$
m = \max_{K_\l}|f_\l| \ge {10}^{-\frac {26}{\d }}\sigma .
$$
Finally, the Bernstein index of $f_\l$ is:
$$
B_{U_\l ,K_\l }(f_\l ) = \log \frac {M(\Lambda )}{m(\lambda )} \le
\log 2 - \log \d + \frac {26}{\d }\log 10 - \log \sigma .
$$
We see that this index, whose estimate took the main part of the
work, is in a sense negligible in comparison with the geometric
factor. Replacing of this index by $|\log \sigma |$ may be well
compensated through the increasing by $1$ the exponential
${10}^5 - 1$ in \eqref{eqn:geom1}.

Finally, by the \gzt we have:
$$
H(2, \d , \sigma , \kappa ) < |\log \sigma |
e^{e^{{10}^5\d^{-3}\b^{-2}}}.
$$
Substituting here the value of $\b $ from Lemma~\ref{lem:main2}
(which is not yet proved), we obtain Theorem \ref{thm:mt2}.

\section{Some lower bounds for trigonometric polynomials}\label{sec:bounds}

In this subsection we complete the proof of  Lemma \ref{lem:main2}.

\subsection{Homogeneous polynomials of degree three} \label{sub:hom}

\begin{Lem}\label{lem:trig} Consider a real homogeneous
trigonometric polynomial $H$ of degree 3, that is, a homogeneous
three--form on $\sin \theta, \cos \theta $ with real coefficients.
Let $\mathbb R_\a$ be the set of all real $\theta$ that are at least
$\a$--distant from the complex rots of $H$. Then
$$
\min_{\mathbb R_{\a}} |H| \ge \frac {\a^3}{24} \| H \|_2.
$$

\end{Lem}
\begin{proof} The polynomial $H$ has three series of roots counted
with multiplicities: $\theta_j + \pi n, n \in \mathbb Z, j = 1,2,3.$
Hence, for some real $A$,
$$
H = A \prod_1^3 \sin (\theta - \theta_j).
$$

\noindent{\it Case} 1. All $\theta_j$ are real. Then
\begin{equation} \label{eqn:minh}
\min_{\mathbb R_\a }|H| \ge |A| \left( \frac {2}{\pi}\right)^3 \a^3.
\end{equation}

On the other hand,
$$
|A| \ge \frac {\|H\|_2}{\sqrt{2 \pi }}.
$$
The inequality:   $ 2^{2.5}/\pi^{3.5} \ge 1/24$ implies the lemma
in Case 1.

\medskip

\noindent{\it Case} 2. One root $\theta_1 $ is real, two others are complex:
$\theta_{2,3} = \varphi \pm i \psi, \psi \not = 0.$ Then
$$
H = A\prod_1^3 \sin (\theta - \theta_j) = \frac A2 \sin (\theta -
\theta_1) (\mbox{ch}2 \psi - \cos 2 (\theta - \varphi)).
$$

For any $a \in \mathbb R, \ |b| \le \pi$, the following inequality
holds:
$$
\mbox{ch}\, a - \cos b \ge \frac 12 a^2 + \left( \frac 2{\pi}\right)^2
b^2.
$$
By assumption, $\psi^2 + {(\theta - \ph )}^2 \ge \a^2$. Hence, once
again we have \eqref{eqn:minh}. This proves the lemma in case 2.
\end{proof}

\subsection{Lower bounds for the distance to the roots} \label{sub:roots}

If two points of the disk $r \le \d^{-1}$ are at least $\d
$--distant in Cartesian coordinates, then they are at least
$\d^2$--distant in the polar coordinates. If two points, one in the
disk $r \le \d^{-1}$ in $\rr^2$, another in $\cc^2$, are at least
$\d $-distant in Cartesian coordinates, $ \d < 0.1$, then they are
at least $\frac 2 3\d^2$-distant in complex polar coordinates.

\begin{Prop} \label{prop:1}   Suppose that the point $ (\theta_0, r), \ r \le
\d^{-1}$ and $\theta_0 \in [0,2\pi ]$ is at least $\frac 2 3
\d^2$--distant from the singular points of the system
\eqref{eqn:rad} with  complexified $r$ and $\theta $ in the metric
$ds^2 = {|dr|}^2 + {|d\theta|}^2$, and
\begin{equation} \label{eqn:est}
\left| r + \frac {1}{g_\l (\theta_0)} \right|< \frac {\d^2}{2}.
\end{equation}
Then
\begin{equation} \label{eqn:alpha1}
d(\theta_0, \{H(v_\l) = 0\} )  \ge \frac{\d^4}{100}.
\end{equation}
\end{Prop}

\begin{proof} By contraposition, assume that the converse to \eqref{eqn:alpha1} is true.
Then there exists $\t_1 $, zero of $H(v_\l )$ \st

$$
|\t_0 - \t_1|< \a := \frac {\d^4}{100}.
$$
It may happen that $\t_1$ is non--real. Take two extra points: $b =
\left( \t_0 , -\dfrac {1}{g_\l (\t_0)}\right) $ and $c = \left(
\t_1, -\dfrac {1}{g_\l (\t_1)}\right) $; and let $a = (\t_0 ,r)$.
Then, by \eqref{eqn:est},
$$
|b - a| \le \frac {\d^2}{2}.
$$
Let $L = \displaystyle\max_{[\t_0, \t_1]}\left|{\left(\dfrac
{1}{g_\l }\right)}'\right|$. By assumption, $|\t_0 -\t_1| \le \a $.
Then, by the Mean Value Theorem,
$$
|b - c| \le \a \sqrt{L^2+1}.
$$
We now estimate $L$ from above. Recall that $g_\l = \mbox{Im }h_\l ,
\ h_\l = Ae^{i\ph } +Be^{-i\ph } + Ce^{-3i\ph }, \ |A| \le 1, |B|
\le 2, |C| \le 1$. By \eqref{eqn:est} and assumption $r \le
\d^{-1}$, we have:
$$
|g_\l (\t_0)| \ge \dfrac {1}{\d^{-1}+\dfrac{\d^2}{2}} \ge \d - \frac
{\d^4}{2} \ge 0.99\d .
$$
Now, by \eqref{eqn:hl}
$$
g_\l =\mbox{Im }h_\l , \qquad |g'_\l | \le |h'_\l | \le |A|e^\a +
|B|e^\a + 3|C|e^{3\a } \le 7.
$$
Then $L < 8\d^{-2}$. Hence,   $\a \sqrt{L^2+1} < \dfrac {\d^2}{6}$.
Therefore, $|a - c| \le |a - b| + |b - c| < \frac 2 3\d^2$, a
contradiction.
\end{proof}


\subsection{Proof of Lemma~\ref{lem:main2}}

In order to prove that no limit cycle can cross $\Pi_\b $, let us
first check that the lower bound $\G^-$ of $\Pi_\b $ is a curve
without contacts with the vector field \eqref{eqn:rad}. This lower
bound has the form:
\begin{equation}    \label{eqn:arc}
\G^-: r = -\frac {1}{g_\l (\theta )} - \b , \ (r,\theta ) \in B(\l ,
\d ),
\end{equation}
Denote by $S$ the minimal slope of the field \eqref{eqn:rad} on
$\G^- $:
$$
S = \min_{\G^- }\left| \frac {dr}{d\theta }\right| = \min_{\G^-
}\left| r\frac {\l_1+rf_\l }{1+rg_\l }\right| .
$$
On $\Gamma^-$ we have:
$$
|\l_1 + rf_\l | = \left| \frac {1}{g_\l }(H(v_\l ) - \beta g_\l f_\l
)\right| \ge \frac 1 4 (|H(v_\l )| - 16\b ).
$$
The points of $B(\l , \d )$ are at least $\d $--distant from the
singular points of system \eqref{eqn:rad}. By the remark at the
beginning of Subsection~\ref{sub:roots}, points of $\G^- $ satisfy
assumptions of Proposition~\ref{prop:1}. Hence, for any $\theta $
\st $(r,\theta ) \in \G^- $ for some $r$, we have
\eqref{eqn:alpha1}. Now, taking $\a = \dfrac{\d^4}{100}$ in
Lemma~\ref{lem:trig}, we conclude that
$$
\min_{\G^- }|H(v_\l )| \ge \frac {\a^3}{24}{||H(v_\l )||}_2 = \frac
{\d^{12}}{{10}^6\cdot 24}{||H(v_\l )||}_2.
$$
By Corollary~\ref{cor:1} we get
$$
\min_{\G^- }|H(v_\l )|\ge \frac {\d^{12}}{{10}^6\cdot 24\sqrt
2}\kappa := \kappa' .
$$
Hence
$$
\min_{\G^- }|\l_1 + rf_\l | \ge \frac {\kappa' }{4} - 4\b .
$$
Moreover, on $\Gamma^-$
$$ |1 + rg_\l | = -\b g_\l \le 4\b . $$
At last, $r|\G^- \ge \dfrac 1 5$. Hence
$$
S \ge \frac {\kappa' }{80\b } - \frac 1 5.
$$

Denote by $\pi \G^- $ the projection of $\G^- $ to $r = 0$ along the
$r$--axis; $\pi \G^- \subset \{ -g^{-1}_\l \le \d^{-1} + \b \}
$. We estimate the maximal slope of $\G^- $. It is equal to
$$
s = \displaystyle \max_{\pi \G^- } \left| {\left(\frac {1}{g_\l
}\right) }'\right| \le \frac {6}{\min_{\pi \G^- }{|g_\l |}^2} \le
7\d^{-2}.
$$
The inequality $S > s$ follows from the definition of $\b $ in
Lemma~\ref{lem:main2}.

We now prove that no $\d $--tame limit cycle that surrounds zero can
cross $\Pi_\b $. On the contrary, let a cycle $\g $ contain a point
$q\in \Pi_\b $. As $\g $ surrounds $0$, it must enter and quit
$\Pi_\b $. The connected component $\Pi^q$ of $\Pi_\b $ that
contains $q$ is bounded by an arc $\g_{\b ,q}$ of the curve
$\eqref{eqn:arc}$ and by the part of $\p B(\l ,\d )$. As $S > s$,
the cycle can enter $\Pi^q$ through $\g_{\b ,q}$   (in positive or
negative time) but cannot quit $\Pi^q$ through $\g_{\b ,q}$. Hence,
it quits $\Pi^q$ through $\p B(\l ,\d )$. This contradicts to the
assumption that $\g $ is $\d $--tame and proves
Lemma~\ref{lem:main2}.

\vskip 1pc

\section{Acknowledgment}

We are grateful to Alexey Fishkin who read several versions of the
manuscript and made many fruitful comments.

\vskip 1pc

\section{The appendix}

In this appendix we provide the values of the $\a$'s, $\b$'s and $\g$'s of Lemma \ref{lem:poin}.

We shall compute the Poincar\'{e} map $P_{\la}$ associated to the
differential equation \eqref{eqn:radc} in complex polar coordinates
$(w,\t)$. Let $P_{\la}:\{\t=0\} \to \{\t=0\}$ be the Poincar\'e map
defined by the flow of system \eqref{eqn:radc}; i.e. $P_{\la}$ is
the $2\pi$--time Poincar\'e map that brings an initial value of any
solution $r(\t,x)$ of system (\ref{eqn:radc}) with initial condition
$r(0,x)= x$ on the half--axis $\{\t=0\}$ to the value of the same
solution at $\t= 2\pi$, whenever defined. We know that the limit
cycles surrounding the origin of system (\ref{eqn:num1}) correspond
to real isolated zeros of the {\it displacement function}
$P_{\la}(x)-x$.

The power series expansion for the displacement function $P_{\la}(x)
- x$ associated to a quadratic system (\ref{eqn:num1}) in a
neighborhood of the origin is found in the following classical way.
The right hand side of equation (\ref{eqn:radc}) may be decomposed
in a power series in $r$ with the $\theta$-dependent coefficients:

\be \frac{dw}{d\t}=\sum_{i=1}^{\infty} R_i(\t)w^i \, , \label{3a}
\ee where $R_1= \la_1$, \be
R_i(\t)=(-1)^i[f_{\l}(\t)-\la_1g_{\l}(\t)]g_{\l}(\t)^{i-2} \quad
\mbox{for} \quad i=2,3,\ldots \label{3b} \ee

The modification of the Bautin result in \cite{Zo} implies that the
coefficients of the displacement map
\begin{equation}\label{popo}
P_{\la}(x)-x= \sum_{j=1}^{\infty} a_j(\la) \, x^j,
\end{equation}
belong to the ideal generated by  $g_j(\la), j = 1,..., 4$, where
$g_j$ are the same as in Theorem \ref{thm:t1}.  For $\la_1=0$, the
coefficients $a_j(\la)$ are polynomial.

We use the algorithm due to Bautin for computing explicitly
$P_{\la}(x)$ in powers of $x$ up to order $7$, see also \cite{CJ}.
We do the computations for the case $\la_1=0$; otherwise
$v_7(\t,\la)$, which is necessary for computing $v_7(2\pi,\la)$ and
consequently $P_{\la}(x)$ in powers of $x$ up to order $7$, would
need more than thousand pages. For doing these computations we have
used the algebraic manipulator mathematica.

We know that the series of (\ref{3a}) converges if $w$ is small
enough, and that the solution $w(\t)$ of differential equation
(\ref{3a}) satisfying the initial condition $w(0)= x$ can be
expanded as
\be w(\t,\la)=\sum_{i=1}^{\infty}  v_i(\t,\la)x^i \, ,
\label{3c} \ee where the $v_i(\t,\la)$'s satisfy the conditions \be
v_1(0,\la)=1 \quad \mbox{and} \quad v_i(0,\la)=0 \quad \mbox{for}
\quad i=2,3,\ldots\, . \label{3d} \ee

Substituting (\ref{3c}) in (\ref{3a}), taking $\la_1=0$, and looking
for the coefficients of the powers of $x$, we obtain the equations
for determining all the $v_i$'s:
\beq
\frac{dv_1}{d\t} &=& 0 \, , \\
\frac{dv_2}{d\t} &=& v_1^2R_2 \, , \\
\frac{dv_3}{d\t} &=& 2v_1v_2R_2+v_1^3R_3 \, , \\
\frac{dv_4}{d\t} &=& (2v_1v_3+v_2^2)R_2+3v_1^2v_2R_3+v_1^4R_4 \, , \\
\frac{dv_5}{d\t} &=&
2(v_1v_4+v_2v_3)R_2+3v_1(v_1v_3+v_2^2)R_3+4v_1^3v_2R_4+v_1^5R_5 \,
,\\
\frac{dv_6}{d\t} &=& (2v_1v_5+2v_2v_4+v_3^2)R_2+(3v_1^2v_4+6v_1v_2v_3+v_2^3)R_3+ \\
 & & 2v_1^2(2v_1v_3+3v_2^2)R_4+5v_1^4v_2R_5+v_1^5R_6 \, ,\\
\frac{dv_7}{d\t} &=& 2(v_1v_6+v_2v_5+v_3v_4)R_2+\\
 & & 3(v_1^2v_5+2v_1v_2v_4+v_1v_3^2+v_2^2v_3)R_3+\\
 & & 4v_1(v_1^2v_4+3v_1v_2v_3+v_2^3)R_4+\\
 & & 5v_1^3(2v_2^2+v_1v_3)R_5+6v_1^5v_2R_6+v_1^7R_7 \, .
\eeq

All these differential equations are solved recursively computing
an integral with respect to $\t$ and taking into account the
initial conditions (\ref{3d}). Thus, we get that $v_1(\t,\la)=1$,
and
\[
\begin{array}{ll}
v_2(\t,\la)=& \dfrac13 (-3 a_2 + 3 b_2 + c_2) + a_2 \cos \t - b_2 \cos \t - \dfrac13  c_2 \cos(3\t) +\\
& a_1 \sin \t + b_1 \sin \t + \dfrac13  c_1 \sin(3\t)\, ,
\end{array}
\]
here we denote $A= a_1+ i a_2$, $B= b_1+i b_2$ and $C= c_1+i c_2$.
The expressions for $v_i(\t,\la)$ for $i=3,4,5,6,7$ need
approximately $1/2$, $2$, $7$, $18$ and $42$ pages, respectively.
Once we know $v_i(\t,\la)$ for $i=3,4,5,6,7$, evaluating
$v_i(2\pi,\la)$ we get the displacement function
\begin{equation}\label{pmap}
P_{\la}(x)= w(2\pi,\la)= \sum_{i=1}^{\infty} v_j(2\pi,\la)x^j =
\sum_{j=1}^{\infty} a_j(\la) \, x^j \, ,
\end{equation}
with the explicit formulas for the polynomials $a_j(\lambda), j =
1,..., 7.$ After that we decompose these polynomials in the ideal
with generators $g_j, j = 1,...,4$. This is done with the use of the
manipulator mathematica again. The results of these computations
presented below imply Lemma \ref{lem:poin}. The coeficients of the
decompositions mentioned above are the following:
\[
\begin{array}{ll}
\a_0=& -2 \pi,\\
 & \\
\b_0=& -\dfrac{2\pi}{3},\\
 & \\
\b_1=& -\dfrac{2\pi}{9} (9 a_2^2-9 b_2 a_2-6
   c_2 a_2-27 \pi  b_1 a_2+27
   b_2^2+21 c_2^2+18 b_1^2+\\
 & \qquad 20 c_1^2+6
   b_2 c_2-27 b_2 \pi  a_1-9 a_1
   b_1),\\
 & \\
\g_0=& -\dfrac{5 \pi }{4},\\
 & \\
\g_1=&  -\dfrac{\pi}{72} (300 a_2^2-558 b_2
   a_2-240 c_2 a_2-384 \pi  b_1 a_2+528 b_2^2+204
   c_2^2-\\
& \quad \quad 36 a_1^2+288 b_1^2+188 c_1^2+168
   b_2 c_2-384 b_2 \pi  a_1-18 a_1 b_1+\\
& \quad \quad 48 a_1 c_1+24 b_1 c_1),\\
 & \\
\g_2=& -\dfrac{\pi}{1080}  (2160 a_2^4-360 b_2 a_2^3-1296
   c_2 a_2^3-25920 \pi  b_1 a_2^3+\\
& \qquad \quad 17100
   b_2^2 a_2^2+27864 c_2^2 a_2^2+21600 \pi
   ^2 b_1^2 a_2^2+10260 b_1^2 a_2^2+\\
& \qquad \quad 24648
   c_1^2 a_2^2+7236 b_2 c_2
   a_2^2-25920 b_2 \pi  a_1 a_2^2-8280 a_1
   b_1 a_2^2+\\
& \qquad \quad 34560 b_2 \pi  b_1
   a_2^2+17280 c_2 \pi  b_1 a_2^2-4752 a_1
   c_1 a_2^2-\\
& \qquad \quad 1740 b_1 c_1
   a_2^2-34200 b_2^3 a_2-19824 c_2^3
   a_2-34560 \pi  b_1^3 a_2-\\
& \qquad \quad 37368 b_2
   c_2^2 a_2+4680 b_2 a_1^2 a_2-144
   c_2 a_1^2 a_2-33480 b_2 b_1^2
   a_2-\\
& \qquad \quad 9954 c_2 b_1^2 a_2+17280 \pi
   a_1 b_1^2 a_2-38472 b_2 c_1^2
   a_2-\\
& \qquad \quad 20976 c_2 c_1^2 a_2-38400 \pi
   b_1 c_1^2 a_2-22806 b_2^2 c_2
   a_2+\\
& \qquad \quad 34560 b_2^2 \pi  a_1 a_2+17280
   b_2 c_2 \pi  a_1 a_2+5040 b_2
   a_1 b_1 a_2+\\
& \qquad \quad 8280 c_2 a_1
   b_1 a_2+43200 b_2 \pi ^2 a_1 b_1
   a_2-60480 b_2^2 \pi  b_1 a_2-\\
& \qquad \quad 41280
   c_2^2 \pi  b_1 a_2-17280 b_2 c_2
   \pi  b_1 a_2-7608
   b_2 a_1 c_1 a_2-\\
& \qquad \quad 1440 c_2
   a_1 c_1 a_2-14628 b_2 b_1
   c_1 a_2-2112 c_2 b_1 c_1
   a_2+36900 b_2^4+\\
& \qquad \quad 13040 c_2^4+14580
   b_1^4+11200 c_1^4+16152 b_2 c_2^3-9720 a_1
   b_1^3+\\
& \qquad \quad 2640 a_1 c_1^3+1320 b_1
   c_1^3+70758 b_2^2 c_2^2-4140 b_2^2
   a_1^2+648 c_2^2 a_1^2+\\
& \qquad \quad 2124 b_2
   c_2 a_1^2+21600 b_2^2 \pi ^2
   a_1^2+50760 b_2^2 b_1^2+42498 c_2^2
   b_1^2-\\
& \qquad \quad 1620 a_1^2 b_1^2-25596 b_2
   c_2 b_1^2-34560 b_2 \pi  a_1 b_1^2+67074
   b_2^2 c_1^2+\\
& \qquad \quad 24000 c_2^2 c_1^2+120
   a_1^2 c_1^2+41670 b_1^2 c_1^2+15288
   b_2 c_2 c_1^2-\\
& \qquad \quad 38400 b_2 \pi  a_1
   c_1^2-17400 a_1 b_1 c_1^2+30996
   b_2^3 c_2-60480 b_2^3 \pi  a_1-\\
& \qquad \quad 41280 b_2
   c_2^2 \pi  a_1-17280 b_2^2 c_2 \pi
   a_1+1080 a_1^3 b_1+17280 b_2 \pi  a_1^2
   b_1-\\
& \qquad \quad 11880 b_2^2 a_1 b_1-19080 c_2^2 a_1
   b_1+540 b_2 c_2 a_1 b_1-720 a_1^3
   c_1-\\
& \qquad \quad 11220 b_1^3 c_1+3690 a_1 b_1^2 c_1+1950 b_2^2
   a_1 c_1+2448 c_2^2 a_1 c_1+\\
& \qquad \quad 1344
   b_2 c_2 a_1 c_1+44028 b_2^2
   b_1 c_1+1224 c_2^2 b_1 c_1+1500
   a_1^2 b_1 c_1+\\
& \qquad \quad 1368 b_2 c_2
   b_1 c_1).
\end{array}
\]

\end{document}